\documentclass[leqno,12pt,a4paper,twoside]{article}%
\usepackage{amssymb}
\usepackage{amsfonts}
\usepackage{amsmath}
\usepackage{authblk}
\usepackage{graphicx}
\usepackage{ifthen}
\usepackage{comment}
\usepackage{xcolor}
\usepackage{enumitem}
\usepackage{lipsum,multicol}
\usepackage[T1]{fontenc}
\usepackage[english]{babel}
\usepackage[english]{babel}
\setlist[itemize]{noitemsep, topsep=0pt}
\setlist[enumerate]{noitemsep, topsep=0pt}
\setlist[itemize]{leftmargin=*}
\setlist[enumerate]{leftmargin=*}
\providecommand{\U}[1]{\protect\rule{.1in}{.1in}}

\providecommand{\norm}[1]{\left\lVert#1\right\rVert}

\providecommand{\pr}[1]{\left(#1\right)} 

\newcommand{\normi}[3]{\norm{#1}_{
		\ifthenelse{\equal{#2}{1}}{H_0^1\pr{\mathcal{O}_{#3}}}{%
			\ifthenelse{\equal{#2}{-1}}{H^{-1}\pr{\mathcal{O}_{#3}}}{}}}}

\makeatletter
\newcommand{\subjclass}[2][2020]{%
	\let\@oldtitle\@title%
	\gdef\@title{\@oldtitle\footnotetext{\textbf{#1 \emph{Mathematics subject classification.}} #2}}%
}
\newcommand{\keywords}[1]{%
	\let\@@oldtitle\@title%
	\gdef\@title{\@@oldtitle\footnotetext{\textbf{\emph{Key words.}} #1.}}%
}
\makeatother

\newtheorem{theorem}{Theorem}

\newtheorem{assumption}[theorem]{Assumption}

\newtheorem{definition}[theorem]{Definition}
\newtheorem{example}[theorem]{Example}

\newtheorem{lemma}[theorem]{Lemma}

\newtheorem{remark}[theorem]{Remark}

\newenvironment{proof}[1][Proof]{\noindent\textbf{#1.} }{\ \rule{0.5em}{0.5em}}

\author[1,*]{Hugo Lhachemi}
\author[2,3,**]{Ionu\c t Munteanu}
\author[4,***]{Christophe Prieur}
\affil[1]{\small Universit\'e Paris-Saclay, CNRS, CentraleSup\`{e}lec, Laboratoire des signaux et syst\`{e}mes, 91190, Gif-sur-Yvette, France} 
\affil[2]{Faculty of Mathematics, Al. I. Cuza University, Bd. Carol I, 11, Iasi 700506, Romania}
\affil[3]{O. Mayer Institute of Mathematics, Romanian Academy, Bd. Carol I, 8, Iasi 700505, Romania}
\affil[4]{Universit\'e Grenoble Alpes, CNRS, Grenoble-INP, GIPSA-lab, F-38000, Grenoble, France}
\affil[*]{e-mail: hugo.lhachemi@centralesupelec.fr}
\affil[**]{e-mail of the corresponding author: ionut.munteanu@uaic.ro}
\affil[***]{e-mail: Christophe.Prieur@gipsa-lab.fr}

\title{Boundary output feedback stabilisation for 2-D and 3-D parabolic equations}

\date{}

\begin{document}
	\maketitle
	\begin{abstract} The present paper addresses the topic of boundary output feedback stabilization of parabolic-type equations, governed by linear differential  operators which can be diagonalized by the introduction of adequate weighting functions (by means of the Sturm-Liouville method), and which evolve in bounded spatial domains that are subsets of $\mathbb{R}^d,\ d=1,2,3$.  
	Combining ideas inspired by \cite{lhachemi2022finite} for the boundary output feedback control of 1-D parabolic PDEs and \cite{munteanu2019boundary} for the state feedback control of multi-D parabolic PDEs, we report in this paper an output feedback boundary stabilizing control with internal Dirichlet measurements designed by means of a finite-dimensional observer. The reported control design procedure is shown to be systematic for 2-D and 3-D parabolic equations.
	\end{abstract}
	
	\section{Introduction}
	 There is now a large amount of literature dealing with boundary control of dynamical infinite-dimensional systems. In particular, various techniques have been developed for the design of control strategies for 1-D partial differential equations (PDEs). These include: Lyapunov methods \cite{bastin:coron:book:2016}, backstepping design \cite{krstic2008boundary}, linear quadratic control methods \cite{curtain2020introduction}, characteristic analysis \cite{bressan:book:2000}, among other approaches. Even if some of these approaches have been generalized to multi-D PDEs, such as LQR methods (see e.g., \cite{van1993h8}), constructive methods for multi-D PDEs are not so developed. Among the  contributions, one can find in~\cite{barbu2013boundary} the design of simple proportional-type boundary stabilizing controllers for multi-D parabolic-type equations under a restrictive assumption concerning the linear independence of the traces of the normal derivatives of the eigenfunctions. Extensions of this approach while removing the aforementioned restrictive assumption have been reported in the recent textbook \cite{munteanu2019boundary}. Boundary control and observation of heat equations on multi-dimensional domains has also been solved in~\cite{feng2022boundary} under the restrictive assumption that all the unstable modes are simple.
	
	In this paper, we study the problem of output feedback stabilization of $2$-D and $3$-D parabolic PDEs using spectral reduction methods. Spectral decomposition techniques consist first of the projection of the PDE into a finite-dimensional unstable system plus a residual stable infinite-dimensional system. In this framework, the control strategy is designed on the unstable finite-dimensional part of the plant. Although simple in their basic concepts, spectral-based control design methods are challenging because one must ensure that the control strategy does not introduce any  spillover effect, that is: the control strategy originally designed on a finite-dimensional approximation of the PDE plant may actually fail to stabilize the infinite-dimensional system due to the interconnection of the controller with the infinite-dimensional residual dynamics (see, e.g., \cite{balas1988finite,balas1978feedback}). These type of approaches originate back to the 60s in the works of Rusell \cite{russell1978controllability} and to the 80s in particular in the work \cite{Triggiani:1980aa}. These ideas were  further developed later in various directions: see e.g. \cite{coron-trelat-heat,lhachemi2020boundary,orlov2000discontinuous} for the one-dimensional case and e.g. \cite{barbu2,munteanu2019boundary,triggiani2007stability,xu2008stabilization} for the multi-dimensional case. 
	 
	Spectral reduction-based methods are very attractive in practice because they allow the design of finite-dimensional control strategies for parabolic PDEs; see for instance the seminal works \cite{Triggiani:1980aa,sakawa1983feedback} in the context of output feedback. They are particularly relevant because they allow the computation of reduced order models, making numerical computations and practical implementations much easier to handle compared to infinite-dimensional control and observation strategies \cite{koga2020materials}. Adopting spectral reduction-based representation of parabolic PDEs~\cite{coron-trelat-heat,russell1978controllability} and leveraging on pioneer works \cite{Triggiani:1980aa,sakawa1983feedback} augmented with Linear Matrix Inequalities (LMIs) procedures~\cite{katz2020constructive}, stabilization problems for 1-D parabolic PDEs have been solved in a systematic manner for various boundary control and boundary output \cite{lhachemi2022finite}, including the possibility to handle systems of 1-D parabolic PDEs \cite{grune2022finite}. In this paper, we further develop and generalize these methods to the case of multi-D parabolic PDEs.

		The aim of this paper is to report a constructive control design procedure for the output feedback boundary stabilization of multi-D parabolic PDEs that is systemic in the 2-D and 3-D cases. Extending procedures reported in \cite{lhachemi2022finite} for 1-D PDEs and combining this latter work with analyses on the eigenspaces inspired by \cite{munteanu2019boundary}, we succeed to design a finite-dimensional observer-based output feedback control strategy for multi-D parabolic-type equations. More precisely, we consider in this work the following boundary-controlled parabolic-type equation evolving in $\mathcal{O}$, an open and connected subset of $\mathbb{R}^d$ with $d\in\{1,2,3\}$, with smooth boundary $\partial\mathcal{O}$ split into two disjoint parts $\partial\mathcal{O}=\Gamma_1\cup\Gamma_2,$ such that $\Gamma_1$ has non-zero Lebesgue measure. Then, the system is described by
	\begin{subequations}\label{e1}
	\begin{align}
	&\partial_tz(x,t)+\sum_{i,j=1}^d a_{ij}(x)\partial_{ij}z(x,t)+\sum_{i=1}^d b_i(x)\partial_iz(x,t)+c(x)z(x,t)=0,\ t>0,\ x\in\mathcal{O};\\
	& z(x,t)=u(x,t),\   x\in \Gamma_1,\ z(x,t)=0 \ x\in\Gamma_2,\ t>0;\\
	& z(0,x)=z_o(x),\ x\in\mathcal{O}.
	\end{align}
	\end{subequations}
The system output consists of $M\in\mathbb{N}^*$ in-domain measurements: 
\begin{equation}\label{measurements}
y(t)=\left(z(\xi_1,t),z(\xi_2,t),\ldots,z(\xi_M,t)\right),
\end{equation} 
with $\xi_i\in\mathcal{O}$ that are 2 by 2 distinct. As we shall see, the number of measurements $M$ to be selected will depend on the maximum multiplicity of the unstable eigenvalues of the plant to be stabilized by the control strategy. In this context, our only assumption in the present work is that the second order governing differential operator can be diagonalized in a suitable Riesz-basis (this will be described in details in Subsection \ref{s2} below). The main result of this paper can be informally stated as follows (see Theorem \ref{main:theorem} for a precise statement)
	
	\noindent \textit{Theorem:  Assuming that the governing linear operator of equation \eqref{e1} is in divergence form (as written in \eqref{eq_div_form} below), there exists an explicit output feedback controller (see \eqref{mio27} below with $U$ provided by \eqref{eq_aux_input_U}) which exponentially stabilizes the reaction-diffusion equation \eqref{e1} based on the sole internal measurement \eqref{measurements}.}
	
	The outline of the paper is as follows. Various notations, assumptions, and preliminary properties are summarized in Section~\ref{s2}. Then, the proposed control stategy along with the main stability result, that rigorously formalizes the above informal theorem, is reported in Section \ref{s3}. Concluding remarks are formulated in Section~\ref{s5}.

	\section{Notation and preliminary properties}\label{s2}
	
	\subsection{Notation and basic definitions}
	
	Spaces $\mathbb{R}^n$ are endowed with the Euclidean scalar product $\left<\cdot,\cdot\right>_n$ and norm $\|\cdot\|_n$. The associated induced norms of matrices are  denoted by $\|\cdot\|$. $L^2(\mathcal{O})$ stands for the space of square Lebesgue integrable functions on $\mathcal{O}$ and is endowed with the inner product $\left<f,g\right>=\int_\mathcal{O}f(x)g(x)dx$ with associated norm denoted by $\|\cdot\|_{L^2}.$ In addition, we denote by $\left<\cdot,\cdot\right>_{L^2(\Gamma_1)}$ the  scalar product in $L^2(\Gamma_1)$ with the Lebesgue surface measure. For an integer $m\geq 1$, the $m-$order Sobolev space is denoted by $H^m(\mathcal{O})$ and is endowed with its usual norm denoted by $\|\cdot\|_{H^m}.$ We set $H_0^1(\mathcal{O})$ for the completion of the space of infinitely differentiable functions, which are nonzero only on a compact subset of $\mathcal{O},$ with respect to the Sobolev norm $\|\cdot\|_{H^1}$. For a symmetric matrix $P\in\mathbb{R}^{n\times n},\ P\succeq 0$ (resp. $P\succ0$) means that $P$ is positive semi-definite (resp. positive definite). 
	
	Let $\left\{\varphi_n\right\},\ n\in\mathbb{N}^*$ be a sequence in a Hilbert space $(H,\left<\cdot,\cdot\right>_{H},\|\cdot\|_{H}).$ It is called a Riesz basis if: (i) $\overline{\text{span}\left\{\varphi_n\right\}}=H$; and (ii) there exists constants $0<c\leq C<\infty$ such that
	$$c\sum_{n\geq 1}|\alpha_n|^2\leq \left\|\sum_{n\geq 1}\alpha_n\varphi_n\right\|_{H}^2\leq C\sum_{n\geq 1}|\alpha_n|^2,$$ 
for all sequences of scalars $\left\{\alpha_n\right\}_n$ so that $\sum_{n\geq1}|\alpha_n|^2<\infty.$
	
	For any given function $\mu\in C(\mathcal{O})$, we introduce the weighted  Lebesgue space 
	$$L^2_\mu(\mathcal{O})=\left\{f:\mathcal{O}\rightarrow\mathbb{R} \;\mathrm{measurable}\, : \, \int_\mathcal{O}f^2(x)\mu(x)dx<\infty\right\}.$$ 
	If there exist constants $0<\mu_m,\mu_M<\infty$ such that $0<\mu_m\leq \mu(x)\leq \mu_M$ almost everywhere, then the two spaces $L^2(\mathcal{O})$ and $L_\mu^2(\mathcal{O})$ are both algebraically and topologically equivalent. This implies, in particular, that a Riesz basis in $L^2(\mathcal{O})$ is also a Riesz basis in $L_\mu^2(\mathcal{O})$ and vice versa. 
	
	\subsection{Differential operator in divergence form}

	Let us denote by $\mathcal{A}: H^2(\mathcal{O})\rightarrow L^2(\mathcal{O})$ the second order differential operator:
	\begin{equation}\label{Op_A}
	\mathcal{A}f=\sum_{i,j=1}^d a_{ij}(x)\partial_{ij}f+\sum_{i=1}^d b_i(x)\partial_if+c(x)f .
	\end{equation} 
	
\begin{assumption}\label{assumption1}
		We assume that there exists a  multiplier $\mu\in C^2(\overline{\mathcal{O}})$, with $0<\mu_m\leq \mu(x)\leq \mu_M$ for all $x\in\mathcal{O}$ for some constants $0<\mu_m,\mu_M<\infty$, such that $\mu\mathcal{A}$ can be rewritten in divergence form:
		\begin{equation}\label{eq_div_form}
		\mu\mathcal{A}f=-\sum_{i=1}^d\partial_i(\tilde{a}_i(x) \partial_i f)+\tilde{c}(x)f
		\end{equation}
		with $C^1(\overline{\mathcal{O}})$-smooth coefficients $\tilde{a}_i,\ \tilde{c}$, for which there exists constants $0 < \tilde{a}_m < \tilde{a}_M$ and $\tilde{c}_m < \tilde{c}_M$ such that $0 < \tilde{a}_m \leq \tilde{a}_i(x) \leq \tilde{a}_M$ and $\tilde{c}_m \leq \tilde{c}(x) \leq \tilde{c}_M$ for all $x\in\overline{\mathcal{O}}$ and all $1 \leq i \leq d$.
\end{assumption}

Previous assumption implies that $-\mathcal{A}_0 := -\mathcal{A}|_{H^2(\mathcal{O})\cap H_0^1(\mathcal{O})}$ is the generator of a $C_0$-analytic semigroup in $L^2(\mathcal{O})$. To show this,  we apply the well-known Hille-Yosida theorem. First, $\mathcal{D}(\mathcal{A}_0)=H^2(\mathcal{O})\cap H_0^1(\mathcal{O})$ is dense in $L^2(\mathcal{O})$. Then, for $\lambda>0$, we consider the equation
	$$(\lambda+\mathcal{A}_0)f = g .$$
	It follows, by scalarly multiplying this equation by $\mu f$ that
	$$\begin{aligned}\left<g,\mu f\right>&=\left<(\lambda+\mathcal{A}_0)f,\mu f\right>\\&
	=\lambda\left<\mu f,f\right>+\left<\mu\mathcal{A}_0f,f\right>\\&
	\geq (\lambda\mu_m+\tilde{c}_m) \|f\|_{L^2(\mathcal{O})}^2+\tilde{a}_m\|\nabla f\|_{L^2(\mathcal{O})}^2.\end{aligned}$$ 
	Hence, for $\lambda$ large enough, we have $\lambda\mu_m + \tilde{c}_m>0$ and
	$$(\lambda\mu_m+\tilde{c}_m)\|f\|_{L^2(\mathcal{O})}^2\leq \mu_M\|g\|_{L^2(\mathcal{O})}\|f\|_{L^2(\mathcal{O})}$$
	or, equivalently
	$$\|(\lambda+\mathcal{A}_0)^{-1}g\|_{L^2(\mathcal{O})}=\|f\|_{L^2(\mathcal{O})}\leq \frac{\mu_M}{\lambda\mu_m+\tilde{c}_m}\|g\|_{L^2(\mathcal{O})}.$$ 
	Thus, $-\mathcal{A}_0$ is the generator of a $C_0$-analytic semigroup  $\left\lbrace e^{-t\mathcal{A}_0}:\ t\geq0\right\rbrace$ in $L^2(\mathcal{O}).$ 
	
	Besides this, the above inequalities, together with the compact embedding of $H_0^1(\mathcal{O})$ in $L^2(\mathcal{O})$, implies the compactness of the resolvent of $-\mathcal{A}_0$. Therefore, $-\mathcal{A}_0$ has a countable set of eigenvalues, which accumulates at infinity, and for which the corresponding eigenfunctions form a Riesz basis in $L^2(\mathcal{O}).$ More exactly, in view of the divergence form \eqref{eq_div_form}, we consider the following weighted eigenvalue problem ($\lambda \in \mathbb{C}$):
	$$-\sum_{i=1}^d\partial_i(\tilde{a}_i\partial_i\varphi)+\tilde{c}\varphi=\mu\lambda\varphi,\ x\in\mathcal{O};\ \varphi=0 \text{ on } \partial\mathcal{O}.$$ 
	Owing to classical theory on spectral properties of elliptic self-adjoint operators (see, e.g. \cite[Chapter 8]{gilbarg1977elliptic}) we know that the above problem has a countable set of solutions formed by an increasing sequence of real eigenvalues $\left\{\lambda_n\right\}_{n\in\mathbb{N}^*}$ which accumulate to infinity and with corresponding eigenfunctions $\left\{\varphi_n\right\}_{n\in\mathbb{N}^*}$ that form an orthonormal basis in $L^2_\mu(\mathcal{O}).$ By the above discussions, we know that $\left\{\varphi_n\right\}_{n\in\mathbb{N}^*}$ forms a Riesz basis in $L^2(\mathcal{O})$ as well (but is not necessarily orthonormal). Let us define $\psi_n:=\mu\varphi_n,\ n\geq 1.$ Then $\left\{\varphi_n\right\}_{n\in\mathbb{N}^*}$ and $\left\{\psi_n\right\}_{n\in\mathbb{N}^*}$ are bi-orthonormal in $L^2(\mathcal{O})$, i.e., 
	$$\left<\varphi_i,\psi_j\right>=\delta_{i,j},\ i,j\geq 1,$$
	where $\delta_{i,j}$ is the Kronecker symbol. In particular, there exist constants $c_1,c_2> 0$ such that 
	\begin{equation}\label{eq_riez_ineq}
	c_1 \sum_{n \geq 1} \left< f , \psi_n \right>^2
	\leq 
	\Vert f \Vert_{L^2(\mathcal{O})}^2
	\leq 
	c_2 \sum_{n \geq 1} \left< f , \psi_n \right>^2 ,
	\quad \forall f \in L^2( \mathcal{O} ).
	\end{equation}	

The following lemma, which is the key ingredient for the introduction of the Lyapunov functional in the proof of our main result stated in Theorem~\ref{main:theorem}, provides a direct relation between the $H_0^1$-norm and the coefficients of projection onto the Riesz basis $\{ \phi_n \}_n$.

\begin{lemma}\label{lemma1}
		Let us fix $\nu \geq 0$ so that $\tilde{c}_m + \nu \mu_m > 0$.  Then, there exist constants $c_3,c_4 > 0$ such that
		\begin{equation}\label{eq_estimate_H1_norm}
		c_3 \Vert f \Vert_{H_0^1(\mathcal{O})}^2
		\leq \sum_{n \geq 1} (\lambda_n + \nu) \left< f , \psi_n \right>^2
		\leq 
		c_4 \Vert f \Vert_{H_0^1(\mathcal{O})}^2,
		\end{equation}
		for all $f \in H^2(\mathcal{O}) \cap H_0^1(\mathcal{O})$.
\end{lemma}	
	
\begin{proof}	Recall that, thanks to the Poincar\'e inequality, we have that $\|\nabla \cdot\|_{L^2(\mathcal{O})}$ is an equivalent norm in $H_0^1(\mathcal{O})$. In particular, there exists a constant $\mathcal{C}>0$ such that 
	$$\|\nabla f\|^2_{L^2(\mathcal{O})}\geq \mathcal{C}\|f\|^2_{H_0^1(\mathcal{O})},\ \forall f\in H_0^1(\mathcal{O}).$$
	
	 Next, 
	we observe for any $f \in H^2(\mathcal{O}) \cap H_0^1(\mathcal{O})$ the following:
	\begin{align*}
	\left< (\mathcal{A}_0+\nu) f , f \right>_{L_\mu^2(\mathcal{O})}
	& = \int_\mathcal{O} \mu \left\{ (\mathcal{A}_0 f) f + \nu f^2 \right\} dx \\
	& = - \sum_{i=1}^d \int_\mathcal{O} \partial_i ( \tilde{a}_i \partial_i f ) f dx + \int_\mathcal{O} (\tilde{c}+\nu\mu) f^2 dx \\
	& = \sum_{i=1}^d \int_\mathcal{O} \tilde{a}_i (\partial_i f )^2 dx + \int_\mathcal{O} (\tilde{c}+\nu\mu) f^2 dx \\
\end{align*}
	where we have applied integration by parts. This shows that $$\left< (\mathcal{A}_0+\nu) f , f \right>_{L_\mu^2(\mathcal{O})}\geq (\tilde{c}_m+\nu\mu_m)\|f\|^2_{L^2(\mathcal{O})},$$ and so $\lambda_n > - \nu$ for all $n \geq 1$; and that
	\begin{equation*}
	 \mathcal{C}\tilde{a}_m\|f\|^2_{H_0^1(\mathcal{O})} \leq \Vert (\mathcal{A}_0+\nu)^{1/2} f \Vert_{L_\mu^2(\mathcal{O})}^2\leq \max(\tilde{a}_M,\tilde{c}_M+\nu\mu_M)\|f\|^2_{H_0^1(\mathcal{O})}.
	\end{equation*}
	Then it can be seen from \eqref{eq_riez_ineq} that
	\begin{align*}
	\sum_{n \geq 1} (\lambda_n + \nu) \left< f , \psi_n \right>^2
	& \leq \frac{1}{c_1} \Vert (\mathcal{A}_0+\nu)^{1/2} f \Vert_{L^2(\mathcal{O})}^2 \\
	& \leq \frac{1}{c_1 \mu_m} \Vert (\mathcal{A}_0+\nu)^{1/2} f \Vert_{L_\mu^2(\mathcal{O})}^2 \\
	& \leq \frac{\max(\tilde{a}_M,\tilde{c}_M+\nu\mu_M)}{c_1 \mu_m} \Vert f \Vert_{H_0^1(\mathcal{O})}^2
	\end{align*}
	while
	\begin{align*}
	\sum_{n \geq 1} (\lambda_n + \nu) \left< f , \psi_n \right>^2
	& \geq \frac{1}{c_2} \Vert (\mathcal{A}_0+\nu)^{1/2} f \Vert_{L^2(\mathcal{O})}^2 \\
	& \geq \frac{1}{c_2 \mu_M} \Vert (\mathcal{A}_0+\nu)^{1/2} f \Vert_{L_\mu^2(\mathcal{O})}^2 \\
	& \geq \frac{\mathcal{C}\tilde{a}_m}{c_2 \mu_M} \Vert f \Vert_{H_0^1(\mathcal{O})}^2 .
	\end{align*}
	This concludes the proof of Lemma \ref{lemma1}.
\end{proof}
	
	We fix the integer $N_0\in\mathbb{N}^*$ such that $\lambda_j\geq \lambda_{N_0 + 1} > 0$ for all $j\geq N_0+1$; and let $N\in\mathbb{N}$ be large enough such that $N\geq N_0$ and $\lambda_n \geq 1,\ \forall n\geq N+1$. A key element in the application of the control strategy reported in this paper relies on the asymptotic behavior of the eigenfunctions $\varphi_n$ evaluated at the measurement locations $\xi_i$. This asymptotic behavior is assessed through the following lemma.
	
\begin{lemma}\label{lemma2}
	 The following holds:
		\begin{equation}\label{eq: lemma 2}
		\sum_{n\geq N+1}\frac{\varphi^2_n(\xi_i)}{\lambda_n^2} < \infty , \qquad \forall i\in\{1,2,\ldots,M\} .
		\end{equation}
\end{lemma}
	
\begin{proof}
	Considering  $\mathcal{O}$ as a manifold with the Riemannian metric $\mu dx$, it follows that $\mathcal{A}$ is a self-adjoint operator in the Lebesgue $L^2$ space associated to this manifold. Hence, one may argue as in \cite{hormander1994spectral} (or as in \cite{sogge2002eigenfunction} for the particular case of the Laplace operator) to deduce the existence of a constant $C>0$ such that, for any $\lambda \geq 1$, we have
	\begin{equation}\label{e2}
	\sum_{\sqrt{\lambda_j}\in[\lambda,\lambda+1)}|\varphi_j(\xi)|^2\leq C\lambda^{d-1},\ \forall \xi\in \overline{\mathcal{O}} .
	\end{equation} 
	
	 Let $0 < \frac{1}{5-d}<\beta<1$, where we recall that $d\in\{1,2,3\}$. We note that $[1,\infty)=\cup_{m\geq 1}[m^\beta,m^\beta+1).$ Therefore, taking $\lambda=m^\beta$ in \eqref{e2}, we infer for any $\xi\in\overline{\mathcal{O}}$ that
	$$\begin{aligned}\sum_{n\geq N+1}\frac{\varphi^2_n(\xi)}{\lambda_n^2}& \leq \sum_{m=1}^\infty \left(\sum_{j \geq N+1,\, \sqrt{\lambda_j}\in[m^\beta,m^\beta+1)}\frac{\varphi^2_j(\xi)}{\lambda_j^2}\right)\leq\sum_{m=1}^\infty\frac{1}{m^{4\beta}}\left(\sum_{j \geq N+1,\,\sqrt{\lambda_j}\in[m^\beta,m^\beta+1)}\varphi^2_j(\xi)\right)\\&\leq C\sum_{m=1}^\infty\frac{1}{m^{4\beta}}m^{(d-1)\beta}
	= C\sum_{m=1}^\infty\frac{1}{m^{(5-d)\beta}}<\infty,
	\end{aligned}$$
since $\beta(5-d)>1.$
\end{proof}
	
\begin{remark}
It is worth being noted that for dimensions $d \geq 4$ the quantities $\sum_{n\geq N+1}\frac{\varphi^2_n(\xi_i)}{\lambda_n^2},\ i\in\left\{1,2,...,M\right\},$ might not be finite, in general. This is the key point that restricts the application of the proposed control strategy for multi-D equations with $d \geq 4$. 
\end{remark}
	
	For latter purpose, we  need to show a unique continuation property of the eigenfunctions of the operator $\mathcal{A}$, as stated in the following lemma.

\begin{lemma}\label{lemma3}
		Let any $\varphi\not\equiv 0$ satisfying
		$$ -\sum_{i=1}^d\partial_i(\tilde{a}_i(x)\partial_i\varphi)+\tilde{c}(x)\varphi-\mu(x)\lambda\varphi=0 \text{ in }\mathcal{O};\ \varphi=0\text{ on }\partial\mathcal{O}.$$ Then, $\sum_{i=1}^dn_i(\cdot)\tilde{a}_i(\cdot)\partial_i\varphi(\cdot)$ is not identically zero on $\Gamma_1$. Here, $n_i$ are the components of the unit outward normal to the boundary of $\mathcal{O}$.
\end{lemma}
	
\begin{proof}	
	This property holds true because the principal part of the differential operator is uniformly elliptic (which, usually, is called the elliptic continuation principle). Let us assume by contradiction that $\sum_{i=1}^dn_i(\cdot)\tilde{a}_i(\cdot)\partial_i\varphi(\cdot)\equiv 0 \text{ on }\Gamma_1.$ Choose some $x_0\in \Gamma_1$, and choose  coordinates $x=(x',x_d)$ so that $x_0=0$ and for some $r>0$
	$$\mathcal{O}\cap B(0,r)=\left\{x\in B(0,r);\ x_d>g(x')\right\},$$ where $g:\mathbb{R}^{d-1}\rightarrow \mathbb{R}$ is a $C^\infty-$function. We extend the domain near $x_0$ by choosing $\psi\in C_c^\infty(\mathbb{R}^{d-1})$ with $\psi=0$ for $\|x'\|_{d-1}\geq r/2$ and $\psi=1$ for $\|x'\|_{d-1}\leq r/4,$ and by letting 
	$$\mathcal{O}^*=\mathcal{O}\cup\left\{x\in B(0,r);\ x_d>g(x')-\varepsilon\psi(x')\right\}.$$ Here, $\varepsilon>0$ is chosen so small that $\left\{(x',x_d);\ \|x'\|_{d-1}\leq r/2,\ x_d=g(x')-\varepsilon\psi(x')\right\}$ is contained in $B(0,r)$. Then, $\mathcal{O}^*$ is connected open set with smooth boundary.
	
	Define the functions
	$$\varphi^*(x)=\left\{\begin{array}{ll}\varphi(x)& \text{ if }x\in\mathcal{O},\\0 & \text{ if }x\in \mathcal{O}^*\setminus\mathcal{O},\end{array}\right.\ $$
	$$\tilde{a}_i^*(x)=\left\{\begin{array}{ll}\tilde{a}_i(x)& \text{ if }x\in\mathcal{O},\\0 & \text{ if }x\in \mathcal{O}^*\setminus\mathcal{O},\end{array}\right.\ $$
	$$(\tilde{c}-\mu\lambda)^*(x)=\left\{\begin{array}{ll}(\tilde{c}-\mu(x)\lambda)(x)& \text{ if }x\in\mathcal{O},\\0 & \text{ if }x\in \mathcal{O}^*\setminus\mathcal{O}.\end{array}\right.\ $$
	Then, $\varphi^*\in H^2(\mathcal{O}^*)$ is solution to
	$$ -\sum_{i=1}^d\partial_i(\tilde{a}^*_i(x)\partial_i\varphi^*(x))+(\tilde{c}-\mu\lambda)(x)^*\varphi^*(x)=0 \text{ a.e. in }\mathcal{O^*},$$with  $\varphi^*\equiv 0$ in some open ball contained in $\mathcal{O}^*\setminus\mathcal{O}.$ Invoking the result in \cite{garofalo1987unique}, we immediately get that $\varphi\equiv0$ in $\mathcal{O}$, which is in contradiction with the hypothesis. It concludes the proof of Lemma \ref{lemma3}.
\end{proof}

With all these elements in hands, we are now in position to introduce the output feedback control strategy that exponentially stabilizes the reaction-diffusion equation \eqref{e1} based on the sole internal measurement \eqref{measurements}. This is discussed in the next section.

\section{Design of the control strategy}\label{s3}

	Let us denote by $A_0:=\left(\left<-\mathcal{A}\varphi_i,\psi_j\right>\right)_{1\leq i,j\leq N_0}.$ It is seen that $A_0=-\text{diag}(\lambda_1,\lambda_2,\dots,\lambda_{N_0}).$ For dimension $d\geq 2$, it is highly possible to have multiple eigenvalues. Note however that the first eigenvalue is always simple. Therefore, the number of scalar measurements $M$ from the system output \eqref{measurements}, as well as the structure of the control strategy, needs to be slightly adapted in function of the multiplicity of the unstable eigenvalues. To fix the ideas and to keep the presentation as concise as possible, we make the following  choice for the spectrum structure (other possible choices can be treated in a similar manner as below without any supplementary effort):   the second eigenvalue is of multiplicity equal two, while the rest of  the first $N_0$ eigenvalues are simple, i.e., we have
	\begin{equation*}
\lambda_1 < \lambda_2 = \lambda_3 < \lambda_4 < \ldots < \lambda_{N_0} .
\end{equation*}	

This configuration leads us to consider $M = 2$ scalar outputs, i.e., 
\begin{equation}\label{measurements case M=2}
y(t)=(z(\xi_1,t),z(\xi_2,t))^\top .
\end{equation}

We then fix $\eta>0$ so that $\lambda_2+\eta\neq \lambda_j,\ \forall j\in\left\{1,3,\dots,N_0\right\}$.

\begin{remark}
Note that the other configurations, in terms of multiplicity of the different eigenvalues, can be handled in a similar way by setting the number of scalar measurements $M$ as the maximum of multiplicity for the eigenvalues $\lambda_1,\ldots,\lambda_{N_0}$. 
\end{remark}

\subsection{Preliminary control design}
	
	For $\gamma>0$ large enough  and for each $v\in L^2(\Gamma_1)$, there exists a unique solution, $D$, to the equation
	\begin{equation}\label{e5}
	\begin{aligned}& \mathcal{A}D-2\sum_{i=1}^{N_0}\lambda_i\left<D,\psi_i\right>\varphi_i-\eta\left<D,\psi_2\right>\varphi_2+\gamma D=0 \text{ in }\mathcal{O}; \\& D=v \text{ on } \Gamma_1,\ D=0 \text{ on }\Gamma_2.
	\end{aligned}
	\end{equation} 
	Indeed, arguing as in \cite{lasiecka1991differential}, this is a direct consequence of the application of the Lax-Milligram theorem. This allows us to introduce the following: 
	
\begin{definition}\label{def operator D}
Let  $D_\gamma: L^2(\Gamma_1)\rightarrow L^2(\mathcal{O})$ be defined by $D_\gamma v:=D$ where, for any given $v\in L^2(\Gamma_1)$, $D \in L^2(\mathcal{O})$ is the unique solution to \eqref{e5}.
\end{definition}	
		
\begin{lemma}
We have
\begin{equation}\label{e6}
	\left(\begin{array}{c}\left<D_\gamma v,\psi_1\right>\\ \left<D_\gamma v,\psi_2\right>\\ \vdots \\ \left<D_\gamma v,\psi_{N_0}\right>\end{array}\right)=-\Lambda_\gamma\left(\begin{array}{c}\left<v,\sum_{i=1}^dn_i\tilde{a}_i\partial_i\varphi_1\right>_{L^2(\Gamma_1)}\\ \left<v,\sum_{i=1}^dn_i\tilde{a}_i\partial_i\varphi_2\right>_{L^2(\Gamma_1)}\\ \vdots \\ \left<v,\sum_{i=1}^dn_i\tilde{a}_i\partial_i\varphi_{N_0}\right>_{L^2(\Gamma_1)}\end{array}\right),
	\end{equation}
	where 
	\begin{equation}\label{eq_Lambda_gamma}
	\Lambda_\gamma 
	= \mathrm{diag}\left(\frac{1}{\gamma-\lambda_1},\frac{1}{\gamma-\lambda_2-\eta},\frac{1}{\gamma-\lambda_3},\dots,\frac{1}{\gamma-\lambda_{N_0}}\right) ,
	\end{equation}	
	while
	\begin{equation}\label{rq1}
	\left<D_\gamma v,\psi_k\right> =- \frac{1}{\gamma+\lambda_k}\left<v,\sum_{i=1}^dn_i\tilde{a}_i\partial_i\varphi_k \right>_{L^2(\Gamma_1)},\ k\geq N_0+1.
	\end{equation}
\end{lemma}		
		
\begin{proof}
Scalarly multiplying equation \eqref{e5} by $\psi_k=\mu\varphi_k,\ k=1,2,\ldots,N_0,$ while taking advantage of the bi-orthogonality of the sequences $\left\{\varphi_i\right\}_{i\geq 1}$  and $\left\{\psi_i\right\}_{i\geq 1}$, we get
	\begin{equation}\label{inteq1}
	\begin{aligned}\left<\mu\mathcal{A}D_\gamma v,\varphi_k\right>-2\lambda_k\left<D_\gamma v,\psi_k\right>-\eta\left<D_\gamma v,\psi_2\right>\delta_{2,k}+\gamma\left<D_\gamma v,\psi_k\right>=0.
	\end{aligned}
	\end{equation}
	Let us compute $\left<\mu\mathcal{A}D_\gamma v,\varphi_k\right>$. Using the integration by parts formula and invoking the boundary conditions of $\varphi_j$ and $D_\gamma v$, we obtain that
	$$\begin{aligned}\left<\mu\mathcal{A}D_\gamma v,\varphi_k\right>&=\left<-\sum_{i=1}^d \partial_i(\tilde{a}_i\partial_iD_\gamma v)+\tilde{c}D_\gamma v,\varphi_k\right>\\&
	= \sum_{i=1}^d \left< D_\gamma v,n_i\tilde{a}_i\partial_i\varphi_k\right>_{L^2(\Gamma_1)}+\left<D_\gamma v,-\sum_{i=1}^d \partial_i(\tilde{a}_i\partial_i\varphi_k)+\tilde{c}\varphi_k\right>\\&
	=\left<v,\sum_{i=1}^dn_i\tilde{a}_i\partial_i\varphi_k\right>_{L^2(\Gamma_1)}+\lambda_k\left<D_\gamma v,\psi_k\right>.\end{aligned}$$
	Substituting the right hand side of the latter identity into equation (\ref{inteq1}), we infer that
	$$(-\lambda_k-\eta\delta_{2,k}+\gamma)\left<D_\gamma v,\psi_k\right>=-\left<v,\sum_{i=1}^dn_i\tilde{a}_i\partial_i\varphi_k\right>_{L^2(\Gamma_1)},\; k\in\{1,\ldots,N_0\}.$$
	This gives \eqref{e6}. Proceeding similarly, we infer that \eqref{rq1} holds.
\end{proof}

	We next fix $N_0$ positive constants $0<\gamma_1<\gamma_2<\ldots<\gamma_{N_0}$, selected sufficiently large such that, for each $k\in\{1,2,\ldots,N_0\}$,
\begin{enumerate}
\item equation \eqref{e5} is well-posed for $\gamma = \gamma_k$; 
\item $\gamma_k\pm(\lambda_i+\eta\delta_{2,i})\neq0$ for all $1\leq k\leq N_0$ and $i\in\mathbb{N}^*$
\end{enumerate}	
	Following Definition~\ref{def operator D}, we denote for $k\in\{1,2,\ldots,N_0\}$ by $D_{\gamma_k}$ the corresponding operators as defined by \eqref{e5}. 
	
	We now introduce the Gram  matrix $\mathbf{B}$, defined by: 
	\begin{equation}\label{ie5}
	\textbf{B}:=\left( \left<\sum_{i=1}^dn_i\tilde{a}_i\partial_i\varphi_k,\sum_{i=1}^dn_i\tilde{a}_i\partial_i\varphi_l\right>_{L^2(\Gamma_1)}\right)_{1\leq k,l\leq N_0} 
	\end{equation} 
	and we set
	\begin{equation}\label{bo}
	B_k:=\Lambda_{\gamma_k} \textbf{B}\Lambda_{\gamma_k},\; k\in\{1,2,\ldots,N_0\}.
	\end{equation}
	We also define 
	\begin{equation}\label{eq_L(x)}
	\mathcal{L}(x)=\left(\sum_{i=1}^dn_i(x)\tilde{a}_i(x)\partial_i\varphi_1(x),\ldots,\sum_{i=1}^dn_i(x)\tilde{a}_i(x)\partial_i\varphi_{N_0}(x)\right)^\top
	\end{equation}	
	for all $x\in\Gamma_1$.	Invoking iteratively the unique continuation property of the eigenfunctions of the operator $\mathcal{A}$, stated by Lemma~\ref{lemma3}, we observe that $\mathcal{L}$ has non-zero entries for all $x$ in a non-zero measure subset of $\Gamma_1$.  Then, because $-\lambda_k-\eta\delta_{2,k}$ for $k\in\{1,2,\dots,N_0\}$ are 2 by 2 distinct and owing to \cite[Proposition 2.1]{munteanu2019boundary}, we have that $B_1+B_2+\ldots+B_{N_0}$ is an invertible matrix. Therefore, we define:
	\begin{equation}\label{ie9}
	A:=(B_1+B_2+\ldots+B_{N_0})^{-1}.
	\end{equation} 
	Next, introducing an auxiliary command input $U:[0,\infty)\rightarrow \mathbb{R}^{N_0},$ that will be specified later in Subsection~\ref{subsec: observer and U}, we define
	\begin{equation}\label{ie10}
	u_k(x,t):= \left<\Lambda_{\gamma_k} AU(t), \mathcal{L}(x)\right>_{N_0} ,\; x\in\Gamma_1 ,\; t\geq0, 
	\end{equation}
	for all $k=1,2,\ldots,N_0$.  Then, we define the boundary control $u$ appearing in the plant \eqref{e1} as 
	\begin{equation}\label{mio27}
	\begin{aligned} u(x,t): &= u_1(x,t)+u_2(x,t)+\ldots+u_{N_0}(x,t) \\&
	= \sum_{k=1}^{N_0}\left<\Lambda_{\gamma_k} AU(t), \mathcal{L}(x)\right>_{N_0}.
	\end{aligned}
	\end{equation}   
	
\begin{lemma}\label{lemma4} It holds, for all $k=1,\ldots,N_0$,
	\begin{equation}\label{e27}\left(\begin{array}{c}\left<D_{\gamma_k}u_k,\psi_1\right>\smallskip\\ \left<D_{\gamma_k}u_k,\psi_2\right>\\\vdots\\\left<D_{\gamma_k}u_k,\psi_{N_0}\right>\end{array}\right)=-B_kAU,\end{equation}
\end{lemma}

\begin{proof}

Based on \eqref{e6}, we infer that
	$$\displaystyle\left( \begin{array}{c}\left<D_{\gamma_k}u_k,\psi_1\right>\\ \left<D_{\gamma_k}u_k,\psi_2\right>\\ \vdots \\ \left<D_{\gamma_k}u_k,\psi_{N_0}\right>\end{array}\right) 
	= - \Lambda_{\gamma_k}
	\left(\begin{array}{c}\left<u_k, \displaystyle\sum_{i=1}^dn_i\tilde{a}_i\partial_i\varphi_1\right>_{L^2(\Gamma_1)}\\ \left<u_k,\displaystyle\sum_{i=1}^dn_i\tilde{a}_i\partial_i\varphi_2\right>_{L^2(\Gamma_1)}\\ \vdots \\ \left<u_k,\displaystyle\sum_{i=1}^dn_i\tilde{a}_i\partial_i\varphi_{N_0}\right>_{L^2(\Gamma_1)}\end{array} \right).$$
	Taking into account \eqref{ie5} and \eqref{ie10}, this yields
	$$\left( \begin{array}{c}\left<D_{\gamma_k}u_k,\psi_1\right>\\ \left<D_{\gamma_k}u_k,\psi_2\right>\\ \vdots \\ \left<D_{\gamma_k}u_k,\psi_{N_0}\right>\end{array}\right) =-\Lambda_{\gamma_k} \mathbf{B}\Lambda_{\gamma_k}  AU$$
	and so, owing to the definition of $B_k$ given by \eqref{bo}, the claimed identity \eqref{e27} is proved.
\end{proof}

\subsection{Spectral reduction}
	
	Our objective is now to specify the auxiliary command input $U$ that appears in \eqref{ie10}. To do so, we first need to carry on a spectral reduction of the system formed by the plant \eqref{e1} along with the preliminary control input \eqref{mio27}. This is done in this subsection.
	
\begin{lemma}
Consider the following change of variable :
	\begin{equation}\label{eq_change_var}
	w:=z-\sum_{k=1}^{N_0}D_{\gamma_k}u_k 
	\end{equation} 
and define the coefficients of projection $w_n(t) = \left< w(t,\cdot) , \psi_n \right>$ and $z_n(t) = \left< z(t,\cdot) , \psi_n \right>$. Then we have
\begin{align}
	\frac{d}{dt}z_n & = - \lambda_n z_n + \sum_{k=1}^{N_0} ( \lambda_n + \gamma_k ) \left< D_{\gamma_k} u_k , \psi_n \right> \nonumber \\
	& \phantom{=}\; - 2 \sum_{k,i=1}^{N_0} \lambda_i \left< D_{\gamma_k} u_k , \psi_i \right> \delta_{i,n} - \eta \sum_{k=1}^{N_0} \left< D_{\gamma_k} u_k , \psi_2 \right> \delta_{2,n} \label{eq_dyn_zn}
\end{align}	
for all $n \geq 1$. Moreover, we have
	\begin{align}
	\frac{d}{dt}w_n & = - \lambda_n w_n + \sum_{k=1}^{N_0} \gamma_k \left< D_{\gamma_k} u_k , \psi_n \right> - \sum_{k=1}^{N_0} \left< D_{\gamma_k} \frac{d}{dt}u_k , \psi_n \right> \label{wop}
	\end{align}
for all $n \geq N_0+1$.
\end{lemma}	
	
\begin{proof}
	We first equivalently rewrite \eqref{e1} as an internal-type control problem. More precisely, invoking the change of variable \eqref{eq_change_var}, we have
	\begin{align}
	\frac{d}{dt}w & =\frac{d}{dt}z-\frac{d}{dt}\sum_{k=1}^{N_0}D_{\gamma_k} u_k \nonumber \\
	& \overset{\eqref{e1}}{=}-\mathcal{A}z-\frac{d}{dt}\sum_{k=1}^{N_0}D_{\gamma_k}u_k \nonumber \\
	& =-\mathcal{A}_0w-\sum_{k=1}^{N_0}\mathcal{A}D_{\gamma_k}u_k-\frac{d}{dt}\sum_{k=1}^{N_0}D_{\gamma_k}u_k \nonumber \\
	& \overset{\eqref{e5}}{=}-\mathcal{A}_0w-2\sum_{k,i=1}^{N_0}\lambda_i\left<D_{\gamma_k}u_k,\psi_i\right>\varphi_i-\eta\sum_{k=1}^{N_0}\left<D_{\gamma_k}u_k,\psi_2\right>\varphi_2 \nonumber \\ 
	& \phantom{=}\; +\sum_{k=1}^{N_0}\gamma_kD_{\gamma_k}u_k-\frac{d}{dt}\sum_{k=1}^{N_0}D_{\gamma_k}u_k ,\ t>0 . \label{eq_dyn_w}
	\end{align}
Then, recalling that $w_n(t) = \left< w(t,\cdot) , \psi_n \right>$ and $z_n(t) = \left< z(t,\cdot) , \psi_n \right>$, the projection of \eqref{eq_dyn_w} gives 
	\begin{align}
	\frac{d}{dt}w_n & = - \lambda_n w_n - \sum_{k=1}^{N_0} \left< D_{\gamma_k} \frac{d}{dt}u_k , \psi_n \right> - 2 \sum_{k,i=1}^{N_0} \lambda_i \left< D_{\gamma_k} u_k , \psi_i \right> \delta_{i,n} \nonumber \\
	& \phantom{=}\; - \eta \sum_{k=1}^{N_0} \left< D_{\gamma_k} u_k , \psi_2 \right> \delta_{2,n} + \sum_{k=1}^{N_0} \gamma_k \left< D_{\gamma_k} u_k , \psi_n \right> \label{eq_dyn_wn}
	\end{align}
for all $n \geq 1$. This gives \eqref{wop} because $\delta_{i,n} = 0$ for all $1 \leq i \leq N_0$ and $n \geq N_0 + 1$. Now, in view of the the change of variable formula \eqref{eq_change_var}, we have
	\begin{equation}\label{eq_wn_zn}
	w_n = z_n - \sum_{k=1}^{N_0} \left< D_{\gamma_k}u_k , \psi_n \right> .
	\end{equation}	
Thus, combining \eqref{eq_dyn_wn}-\eqref{eq_wn_zn}, we infer that \eqref{eq_dyn_zn} holds.
\end{proof}
	
Our objective is now to write an ODE describing the dynamics of a finite number of modes of the system composed of the plant \eqref{e1} and the control input \eqref{mio27}. To do so, we define 
\begin{equation*}
Z^{N_0}:=(\left<z,\psi_1\right>,\left<z,\psi_{2}\right>,\ldots,\left<z,\psi_{N_0}\right>)^\top
\end{equation*}
and
\begin{equation*}
\Xi:=\mathrm{diag}(0,\eta,0,\dots,0) .
\end{equation*}
In view of equation \eqref{eq_dyn_zn} and taking into account the relation \eqref{e27}, we deduce that
	\begin{align}
	\frac{d}{dt}Z^{N_0}(t) & = A_0 Z^{N_0}(t) + \sum_{k=1}^{N_0}\left[-A_0+\gamma_k I\right]\left( \begin{array}{c}\left<D_{\gamma_k}u_k,\psi_1\right>\\ \left<D_{\gamma_k}u_k,\psi_2\right>\\ \ldots\\ \left<D_{\gamma_k}u_k,\psi_{N_0}\right>\end{array}\right) \nonumber \\
	& \phantom{=}\; +\sum_{k=1}^{N_0}2A_0\left( \begin{array}{c}\left<D_{\gamma_k}u_k,\psi_1\right>\\ \left<D_{\gamma_k}u_k,\psi_2\right>\\ \ldots\\ \left<D_{\gamma_k}u_k,\psi_{N_0}\right>\end{array}\right)
	-\sum_{k=1}^{N_0}\Xi\left( \begin{array}{c}\left<D_{\gamma_k}u_k,\psi_1\right>\\ \left<D_{\gamma_k}u_k,\psi_2\right>\\ \ldots\\ \left<D_{\gamma_k}u_k,\psi_{N_0}\right>\end{array}\right) \nonumber \\
	& = A_0 Z^{N_0}(t) -\sum_{k=1}^{N_0}\left[A_0+\gamma_k I-\Xi\right]B_kAU(t)  \nonumber \\
	& \overset{\eqref{ie9}}{=} A_0 Z^{N_0}(t) + \left( - A_0 - \sum_{k=1}^{N_0} \gamma_k B_k A + \Xi \right) U(t) . \label{op2}
	\end{align} 
	where we recall that $A_0=-\text{diag}(\lambda_1,\ldots,\lambda_{N_0})$. Introducing now a second integer $N \geq N_0 + 1$ to be specified later, we define 
	\begin{equation*}
	Z^{N-N_0}:=(\left<z,\psi_{N_0+1}\right>,\left<z,\psi_{N_0+2}\right>,\ldots,\left<z,\psi_{N}\right>)^\top
	\end{equation*}
	along with $A_1=-\text{diag}(\lambda_{N_0+1},\ldots,\lambda_N)$ and $H^{N-N_0} : \mathbb{R}^{N_0} \rightarrow \mathbb{R}^{N-N_0}$ defined by 
	\begin{align*}
	H^{N-N_0}U & = \left(\sum_{k=1}^{N_0}( \lambda_{N_0 + 1} + \gamma_k ) \left<D_{\gamma_k}u_k ,\psi_{N_0+1}\right>,\ldots,\sum_{k=1}^{N_0} ( \lambda_N + \gamma_k ) \left<D_{\gamma_k}u_k ,\psi_{N}\right>\right)^\top \\
	& \overset{\eqref{rq1}}{=} 
	\sum_{k=1}^{N_0}\left(\begin{array}{c}\left<u_k,\sum_{i=1}^dn_i\tilde{a}_i\partial_i\varphi_{N_0+1}\right>_{L^2(\Gamma_1)}\\
	\left<u_k,\sum_{i=1}^dn_i\tilde{a}_i\partial_i\varphi_{N_0+2}\right>_{L^2(\Gamma_1)}\\ \ldots \\
	\left<u_k,\sum_{i=1}^dn_i\tilde{a}_i\partial_i\varphi_{N}\right>_{L^2(\Gamma_1)}\end{array}\right)
	\end{align*}
where $u_k$ is expressed in function of the auxiliary control input $U$ based on \eqref{ie10}. Then, using again \eqref{eq_dyn_zn}, we deduce that
	\begin{equation}\label{oop2}
	\frac{d}{dt} Z^{N-N_0}(t) = A_1 Z^{N-N_0}(t) + H^{N-N_0}U(t) ,\; t>0.	
	\end{equation} 

\subsection{Observer design and definition of the auxiliary control input $U$}\label{subsec: observer and U}
		
	We are now in position to properly define in this subsection the auxiliary command input $U$ that appears in \eqref{ie10}. 
We select the measurement locations $\xi_1,\xi_2\in\mathcal{O}$ from \eqref{measurements case M=2} such that 
	\begin{equation}\label{output_cond}
	\vert \varphi_i(\xi_1) \vert + \vert \varphi_i(\xi_2) \vert \neq 0 , \; \forall i\in\{1,4,5,\ldots,N_0\} ;
	\qquad 
	\det\left|\begin{array}{cc}\varphi_2(\xi_1)&\varphi_3(\xi_1)\\
	\varphi_2(\xi_2)&\varphi_3(\xi_2)\end{array}\right|\neq0 .
	\end{equation} 
Note that such a selection is always possible due to the fact that $\varphi_i$, $i\in\{1,2,\ldots,N_0\}$, cannot vanish on any open ball included in $\mathcal{O}$ and the fact that $\varphi_2,\varphi_3$ correspond to the same eigenvalue and are linearly independent. 
	
	Second, introducing $\tilde{y}(t)=(w(\xi_1,t),w(\xi_2,t))^\top$, we deduce from the change of variable formula \eqref{eq_change_var} that
	\begin{equation}\label{output}
	y(t) 
	= \tilde{y}(t) + \sum_{k=1}^{N_0} \begin{pmatrix}  D_{\gamma_k} u_k(\xi_1,t) \\ D_{\gamma_k }u_k (\xi_2,t) \end{pmatrix}
	= \sum_{i \geq 1} \begin{pmatrix} \varphi_i(\xi_1) \\ \varphi_i(\xi_2) \end{pmatrix} w_i(t)
	+ \sum_{k=1}^{N_0} \begin{pmatrix}  D_{\gamma_k} u_k(\xi_1,t) \\ D_{\gamma_k }u_k (\xi_2,t) \end{pmatrix} .
	\end{equation}
	Defining 
	$C_0 = \left( \begin{array}{c}\varphi_1(\xi_1) \ \varphi_2(\xi_1)\ \ldots \ \varphi_{N_0}(\xi_1)\\
	\varphi_1(\xi_2) \ \varphi_2(\xi_2)\ \ldots \ \varphi_{N_0}(\xi_2)\end{array}\right)$ 
	and 
	$C_1=\left(\begin{array}{c}\varphi_{N_0 + 1}(\xi_1) \ \varphi_{N_0+2}(\xi_1) \ldots \ \varphi_{N}(\xi_1)\\
	\varphi_{N_0 + 1}(\xi_2) \ \varphi_{N_0+2}(\xi_2) \ldots \ \varphi_{N}(\xi_2)\end{array}\right)$,
	it immediately follows from \eqref{output_cond} that the pair $(A_0,C_0)$ satisfies the Kalman condition. Hence, we can fix $L\in M_{N_0\times 2}(\mathbb{R})$ such that $A_0-LC_0$ is Hurwitz with arbitrary spectral abscissa.
	
	In view of \eqref{eq_wn_zn}, \eqref{op2}, and \eqref{oop2}, we now define the following observer dynamics:
	\begin{subequations}\label{op30}
		\begin{align}
		\hat{w}_n & = \hat{z}_n - \sum_{k=1}^{N_0} \left< D_{\gamma_k}u_k , \psi_n \right> ,\quad 1 \leq n \leq N \\
		\frac{d}{dt}\hat{Z}^{N_0}(t) & = A_0 \hat{Z}^{N_0}(t) + \left( - A_0 - \sum_{k=1}^{N_0} \gamma_k B_k A + \Xi \right) U(t) \\ 
		& \phantom{=}\; - L \left\{ \sum_{i=1}^N \begin{pmatrix} \varphi_i(\xi_1) \\ \varphi_i(\xi_2) \end{pmatrix} \hat{w}_i(t) + \sum_{k=1}^{N_0} \begin{pmatrix} \left( D_{\gamma_k u_k} \right)(\xi_1,t) \\ \left( D_{\gamma_k u_k} \right)(\xi_2,t) \end{pmatrix} - y(t) \right\} \nonumber \\
		\frac{d}{dt}\hat{Z}^{N-N_0}(t) & = A_1\hat{Z}^{N-N_0}(t)+H^{N-N_0}U(t) 
		\end{align}
	\end{subequations}
	for $t > 0$. This allows us to complete the definition of the control strategy by setting the auxiliary control input $U$ as:
	\begin{equation}\label{eq_aux_input_U}
	U=\hat{Z}^{N_0} .
	\end{equation}
	Overall, the control strategy is composed of \eqref{mio27}, \eqref{op30}, and \eqref{eq_aux_input_U}.

\subsection{Main stabilization result}	
	
We can now state the main result of this work.

\begin{theorem}\label{main:theorem}
Assume that Assumption~\ref{assumption1} holds. Let $\delta>0$ and $N_0\geq 1$ be such that $\lambda_n>\delta$ for all $n\geq N_0+1$. Assume that the first $N_0$ eigenvalues of the operator $\mathcal{A}$ are simple except of the second and the third one which are equal.  With corresponding measurement \eqref{measurements case M=2}, pick $\xi_1,\xi_2\in\mathcal{O}$ so that \eqref{output_cond} holds true.  Let $L\in\mathbb{R}^{N_0} $ be such that $A_0-LC_0$ is Hurwitz with eigenvalues that have a real part strictly less than $-\delta$. Then, $\gamma_1<\gamma_2<\ldots<\gamma_{N_0}$ can be selected large enough such that the matrix $-\sum_{k=1}^{N_0}\gamma_k B_kA+\Xi$ is Hurwitz with eigenvalues that have a real part strictly less than $-\delta$. Furthermore, for $N\geq N_0+1$ selected to be large enough,
there exists a constant $C>0$ such that, for any initial condition $z_o\in H^2(\mathcal{O})$,  the  trajectory of the closed-loop system composed of the plant \eqref{e1}, the internal measurement \eqref{measurements} and the controller \eqref{mio27} with $U$ given by \eqref{eq_aux_input_U}, satisfies
\begin{equation}\label{eq: main stability estimate}
\Vert z(\cdot,t) \Vert_{H^1(\mathcal{O})} + \sum_{k=1}^N \vert \hat{z}_k(t) \vert 
\leq C e^{-\delta t} \left( \Vert z(0,t) \Vert_{H^1(\mathcal{O})} + \sum_{k=1}^N \vert \hat{z}_k(0) \vert \right).
\end{equation}
\end{theorem}	
	
\begin{example}
	Consider the plant \eqref{e1} with $d=2$, $\mathcal{O}=(0,\pi)\times(0,\pi)$, $\Gamma_1=\left\{x_1\in(0,\pi),\ x_2=0\right\}$, and $\mathcal{A}f=-\Delta f-3\nabla\cdot f-10f$. Hence, introducing $\mu=e^{3x_1+3x_2}$, which is positive and bounded, $1\leq \mu(x)\leq e^{6\pi},$ Assumption~\ref{assumption1} is fulfilled because:
	$$\mu\mathcal{A}f=-\sum_{i=1}^2\partial_i(e^{3x_1+3x_2}\partial_i f)-10e^{3x_1+3x_2}f.$$ 
	The associated eigenfunctions are described by $\varphi_{i,j}(x)=\frac{2}{\pi}e^\frac{-3x_1-3x_2}{2}\sin(ix_1)\sin(jx_2)$ with the corresponding eigenvalues $\lambda_{i,j}=\frac{4i^2+4j^2-31}{4}$. Then, it is seen that the bi-orthogonal system is given by $\psi_{i,j}=\frac{2}{\pi}e^\frac{3x_1+3x_2}{2}\sin(ix_1)\sin(jx_2)$, $i,j\in\mathbb{N}\setminus\left\{0\right\}$. 
	
	In this setting, the open-loop system is unstable with a total of three modes that are not exponentially stable: $-\lambda_{1,1}=23/4$ and $-\lambda_{1,2}=-\lambda_{2,1}=11/4$ with the corresponding three eigenfunctions 
	$$\left\{\frac{2}{\pi}e^\frac{-3x_1-3x_2}{2}\sin x_1\sin x_2, \ \frac{2}{\pi}e^\frac{-3x_1-3x_2}{2}\sin 2x_1\sin x_2,\ \frac{2}{\pi}e^\frac{-3x_1-3x_2}{2}\sin x_1\sin 2x_2\right\}.$$
Hence, fixing $N_0 = 3$, the choice for the spectral multiplicity, of the first three eigenvalues, is verified. Moreover, \eqref{output_cond} holds true as soon as $\xi_i=(\xi_{i1},\ \xi_{i2})\in (0,\pi)^2,\ i=1,2,$ such that $\cos\xi_{11}\cos\xi_{22}-\cos\xi_{12}\cos\xi_{21}\neq0$. This allows the application of Theorem~1.
\end{example}

\begin{proof} 
We first need to rewrite the dynamics of the closed-loop system formed by \eqref{mio27}, \eqref{op30}, and \eqref{eq_aux_input_U} in a suitable format for the upcoming stability analysis. To do so, we first define the errors of observation as
	$$E^{N_0}:=Z^{N_0}-\hat{Z}^{N_0}, \quad \tilde{E}^{N-N_0}:=\Lambda^{N-N_0}(Z^{N-N_0}-\hat{Z}^{N-N_0})$$
	where $\Lambda^{N-N_0}=\text{diag}(\lambda_{Q_0+1},\ldots,\lambda_{N})$ while
	$$\zeta := \sum_{n\geq N+1} \begin{pmatrix} \varphi_n(\xi_1) \\ \varphi_n(\xi_2) \end{pmatrix} w_n, \quad \tilde{C}_1 := C_1 \left(\Lambda^{N-N_0}\right)^{-1} .$$
	Hence, we obtain from \eqref{mio27}, \eqref{op30}, and \eqref{eq_aux_input_U} that
	\begin{subequations}\label{e8}
		\begin{align}
		\frac{d}{dt}\hat{Z}^{N_0} & = \left( -\sum_{k=1}^{N_0}\gamma_k B_k A + \Xi \right) \hat{Z}^{N_0} + L C_0 E^{N_0} + L \tilde{C}_1 \tilde{E}^{N-N_0} + L \zeta, \\
		\frac{d}{dt}E^{N_0} & = (A_0-LC_0)E^{N_0}-L\tilde{C}_1\tilde{E}^{N-N_0}-L\zeta , \\
		\frac{d}{dt}\hat{Z}^{N-N_0} & = A_1\hat{Z}^{N-N_0}+H^{N-N_0}\hat{Z}^{N_0}, \\
		\frac{d}{dt}\tilde{E}^{N-N_0} & = A_1\tilde{E}^{N-N_0}.
		\end{align}
	\end{subequations}
	Introducing the finite-dimensional state vector $X=(\hat{Z}^{N_0},E^{N_0},\tilde{E}^{N-N_0})^\top$ along with the matrices
	\begin{equation}\label{eq_matrix_F}
	F:= 
	\begin{pmatrix}
	-\sum_{k=1}^{N_0}\gamma_kB_kA+\Xi & L C_0 & L \tilde{C}_1 \\
	0 & A_0 - L C_0 & -L \tilde{C}_1 \\
	0 & 0 & A_1
	\end{pmatrix} , \quad
	\mathcal{L} := 
	\begin{pmatrix}
	L \\ -L \\ 0 
	\end{pmatrix} ,
	\end{equation}	
	we infer that the closed-loop system dynamics is described by
	\begin{subequations}\label{eq_CL_traj}
		\begin{align}
		\frac{d}{dt}X & = FX+\mathcal{L}\zeta \\
		\frac{d}{dt}\hat{Z}^{N-N_0} & = A_1\hat{Z}^{N-N_0}+H^{N-N_0}\hat{Z}^{N_0} \\
		\frac{d}{dt}w_n & = - \lambda_n w_n + \sum_{k=1}^{N_0} \gamma_k \left< D_{\gamma_k} u_k , \psi_n \right> - \sum_{k=1}^{N_0} \left< D_{\gamma_k} \frac{d}{dt}u_k , \psi_n \right> , \quad n \geq N+1
		\end{align}
	\end{subequations}

Let $c > 1$ be an arbitrarily given constant. Let us now show that we can fix the real numbers $0<\gamma_1<\gamma_2<\ldots<\gamma_{N_0}< c \gamma_1$ large enough such that the matrix  $-\sum_{k=1}^{N_0}\gamma_kB_kA+\Xi$ is Hurwitz with eigenvalues that have a real part strictly less than $-\delta < 0$. To do so, let $\lambda\in\mathbb{C}$ and a non-zero vector $Z\in\mathbb{C}^{N_0}$ be such that $\left(-\sum_{k=1}^{N_0}\gamma_k B_k A + \Xi \right) Z=\lambda Z$. Recalling that $A$ is defined by \eqref{ie9}, $A$ is symmetric definite positive and $\sum_{k=1}^{N_0} B_k A = I$. Hence, it follows that 
	\begin{align*}
	\lambda\|A^\frac{1}{2}Z\|^2_{N_0} 
	& =\left<\lambda Z,AZ\right>_{N_0}=-\sum_{k=1}^{N_0}\gamma_k\left<B_kAZ,AZ\right>_{N_0}+\left<\Xi Z,AZ\right>_{N_0} \\
	& = -\gamma_1\left<Z,AZ\right>_{N_0}+\gamma_1\sum_{k=1}^{N_0}\left<B_kAZ,AZ\right>_{N_0}-\sum_{k=1}^{N_0}\gamma_k\left<B_kAZ,AZ\right>_{N_0} \\
	& \phantom{=}\; +\left<A^\frac{1}{2}\Xi A^{-\frac{1}{2}}A^\frac{1}{2}Z,A^\frac{1}{2}Z\right>_{N_0}\\&=-\gamma_1\|A^\frac{1}{2}Z\|^2_{N_0}+\sum_{k=2}^{N_0}(\gamma_1-\gamma_k)\left<B_kAZ,AZ\right>_{N_0}+\left<A^\frac{1}{2}\Xi A^{-\frac{1}{2}}A^\frac{1}{2}Z,A^\frac{1}{2}Z\right>_{N_0} .
	\end{align*}
	We have for all $k \in\{1,\ldots,N_0\}$ that $\gamma_1 \leq \gamma_k$ and $B_k$ is positive semi-definite (in view of its definition \eqref{bo}).  Moreover, by the definition of $A$, using the Landau notation, we have
	\begin{equation*}
	\|A^\frac{1}{2}\| = O(\gamma_1)
	, \quad
	\|A^{-\frac{1}{2}}\| = O\left(\frac{1}{\gamma_1}\right)
	\end{equation*}
	as $\gamma_1 \rightarrow + \infty$. Hence $\|A^\frac{1}{2}\Xi A^{-\frac{1}{2}}\|\leq C\eta,$ for some constant $C>0$ independent of $0<\gamma_1<\gamma_2<\ldots<\gamma_{N_0} < c \gamma_1$. It then yields from the above that
	$$\Re(\lambda)\|A^\frac{1}{2}Z\|^2_{N_0}\leq (-\gamma_1+C\eta)\|A^\frac{1}{2}Z\|_{N_0}^2.$$ 
	Since $\|A^\frac{1}{2}Z\|_{N_0} \neq 0$, we obtain for $\gamma_1$ large enough that $\Re(\lambda)<-\delta$, which proves our claim. This, in particular, implies that the matric $F$ defined by \eqref{eq_matrix_F} is Hurwitz with eigenvalues that have a real part strictly less than $-\delta$.

	We now carry on a Lyapunov stability analysis. In view of \eqref{eq_estimate_H1_norm}, let us introduce the Lyapunov function, for all $(X,w)\in \mathbb{R}^{N+N_0}\times H^1(\mathcal{O})$
	\begin{equation}\label{eq_Lyap_fun} 
	V(X,w) = X^\top P X + \sum_{n \geq N+1} (\lambda_n+\nu) w_n^2 .
	\end{equation}
	The computation of the time derivative of $V$ along the system trajectories \eqref{eq_CL_traj} gives
	\begin{align*}
	\dot{V} 
	& = 2 X^\top P ( F X + \mathcal{L} \zeta ) \\
	& \phantom{=}\, + 2 \sum_{n \geq N+1} (\lambda_n+\nu) \left\{ - \lambda_n w_n + \sum_{k=1}^{N_0} \gamma_k \left< D_{\gamma_k} u_k , \psi_n \right> - \sum_{k=1}^{N_0} \left< D_{\gamma_k} \frac{d}{dt}u_k , \psi_n \right> \right\} w_n \\
	& = \tilde{X}^\top \begin{pmatrix} F^\top P + P F & P \mathcal{L} \\ \mathcal{L}^\top P & 0 \end{pmatrix} \tilde{X} - 2 \sum_{n \geq N+1} \lambda_n (\lambda_n+\nu) w_n^2 \\
	& \phantom{=}\; + 2 \sum_{n \geq N+1} (\lambda_n+\nu) \sum_{k=1}^{N_0} \gamma_k \left< D_{\gamma_k} u_k , \psi_n \right> w_n
	- 2 \sum_{n \geq N+1} (\lambda_n+\nu) \sum_{k=1}^{N_0} \left< D_{\gamma_k} \frac{d}{dt}u_k , \psi_n \right> w_n
	\end{align*}
	where $\tilde{X} : = \mathrm{col}(X,\zeta)$. Let us estimate the before last term. To do so, let us note that
	\begin{align*}
	\left< D_{\gamma_k} u_k , \psi_n \right>
	& \overset{\eqref{ie10}}{=} - \left< D_{\gamma_k} \left<\Lambda_{\gamma_k} AU, L\right>_{N_0} , \psi_n \right> \\
	& = - \sum_{l=1}^{N_0} \tilde{\gamma}_{k,l} A_{L_l} U \left< D_{\gamma_k} L_l , \psi_n \right>
	\end{align*}
	where $\tilde{\gamma}_{k,l}$ in the $l$th term on the diagonal of the matrix $\Lambda_{\gamma_k}$ defined by \eqref{eq_Lambda_gamma}, $A_{L_l}$ is the $l$th line of the matrix $A$, and $L_l(x)$ is the $l$th component of $L(x)$ defined by \eqref{eq_L(x)}. Therefore, using Young's inequality, we deduce for any $\epsilon > 0$ that 	
	\begin{align*}
	& 2 \sum_{n \geq N+1} (\lambda_n+\nu) \sum_{k=1}^{N_0} \gamma_k \left< D_{\gamma_k} u_k , \psi_n \right> w_n \\
	& = - 2 \sum_{n \geq N+1} (\lambda_n+\nu) \sum_{k,l=1}^{N_0} \gamma_k \tilde{\gamma}_{k,l} A_{L_l} U \left< D_{\gamma_k} L_l , \psi_n \right> w_n \\
	& \leq 2 \sum_{k,l=1}^{N_0} \sum_{n \geq N+1} \vert \gamma_k \vert \vert \tilde{\gamma}_{k,l} \vert \Vert A_{L_l} \Vert \vert \left< D_{\gamma_k} L_l , \psi_n \right> \vert \Vert U \Vert \times (\lambda_n+\nu) \vert w_n \vert \\
	& \leq \epsilon \sum_{k,l=1}^{N_0} \gamma_k^2 \tilde{\gamma}_{k,l}^2 \Vert A_{L_l} \Vert^2 \left\{ \sum_{n \geq N+1} \left< D_{\gamma_k} L_l , \psi_n \right>^2 \right\} \Vert U \Vert^2
	+ \frac{1}{\epsilon} \sum_{k,l=1}^{N_0} \sum_{n \geq N+1} (\lambda_n+\nu)^2 w_n^2 \\
	& \overset{\eqref{eq_riez_ineq}}{\leq} \epsilon \underbrace{c_1 \sum_{k,l=1}^{N_0} \gamma_k^2 \tilde{\gamma}_{k,l}^2 \Vert A_{L_l} \Vert^2 \Vert \mathcal{R}_N D_{\gamma_k} L_l \Vert^2}_{:=S_{1,N}} \Vert U \Vert^2
	+ \frac{N_0^2}{\epsilon} \sum_{n \geq N+1} (\lambda_n+\nu)^2 w_n^2
	\end{align*}
	where $\mathcal{R}_N f := \sum_{n \geq N+1} \left< f , \psi_n \right> \varphi_n$ and $U = \hat{Z}^{N_0} = E_1 X$ with $E_1 = \begin{bmatrix} I & 0 & 0 \end{bmatrix}$. Similarly, we have
	\begin{align*}
	& 2 \sum_{n \geq N+1} (\lambda_n+\nu) \sum_{k=1}^{N_0} \left< D_{\gamma_k} \frac{d}{dt} u_k , \psi_n \right> w_n \\
	& \overset{\eqref{eq_riez_ineq}}{\leq} \epsilon \underbrace{c_1 \sum_{k,l=1}^{N_0} \tilde{\gamma}_{k,l}^2 \Vert A_{L_l} \Vert^2 \Vert \mathcal{R}_N D_{\gamma_k} L_l \Vert^2}_{:=S_{2,N}} \left\Vert \frac{d}{dt} U \right\Vert^2
	+ \frac{N_0^2}{\epsilon} \sum_{n \geq N+1} (\lambda_n+\nu)^2 w_n^2 
	\end{align*}
	where $\frac{d}{dt}U = \frac{d}{dt}\hat{Z}^{N_0} = E_2 \tilde{X}$ with $E_2 = \begin{bmatrix} -\sum_{k=1}^{N_0}\gamma_kB_kA+\Xi & L C_0 & L \tilde{C}_1 & L \end{bmatrix}$. Finally, we have by Cauchy-Schwarz inequality that
	\begin{align*}
	\Vert \zeta \Vert^2
	& = \left( \sum_{n \geq N+1} \varphi_n(\xi_1) w_n \right)^2 + \left( \sum_{n \geq N+1} \varphi_n(\xi_2) w_n \right)^2 \\
	& \leq \underbrace{\sum_{n \geq N+1} \frac{\varphi_n(\xi_1)^2 + \varphi_n(\xi_2)^2}{(\lambda_n+\nu)^2}}_{:= S_{\varphi,N}} \times \sum_{n \geq N+1} (\lambda_n+\nu)^2 w_n^2 .
	\end{align*}
	Gathering all the the above estimates, we deduce that
	\begin{equation}\label{eq: dV/dt}
	\dot{V} + 2 \delta V \leq \tilde{X}^\top \Theta_1 \tilde{X} + \sum_{n \geq N+1} (\lambda_n+\nu) \Psi_n w_n^2 
	\end{equation}
	for $\delta >0$ fixed such that $F + \delta I$ is Hurwitz and where
	\begin{align*}
	\Theta_1 & = \begin{pmatrix} F^\top P + P F + 2 \delta P + \epsilon S_{1,N} E_1^\top E_1 & P \mathcal{L} \\\mathcal{L}^\top P & -\eta I \end{pmatrix} + \epsilon S_{2,N} E_2^\top E_2 , \\
	\Psi_n & = \left[ - 2 \left( 1 - \frac{N_0^2}{\epsilon} \right) + \eta S_{\varphi,N} \right] \lambda_n + \left[ \frac{2 N_0^2}{\epsilon} + \eta S_{\varphi,N} \right] \nu + 2 \delta , \quad n \geq N+1 
	\end{align*}
	for an arbitrary $\eta > 0$. 
	
	Assume for the moment that $\Theta_1 \preceq 0$ and $\Psi_n \leq 0$ for all $n \geq N+1$. Then, in view of \eqref{eq: dV/dt}, we get that $\dot{V}+2\delta V \leq 0$. Combining this estimate with the definition \eqref{eq_Lyap_fun} of the Lyapunov function $V$, the direct integration of the dynamics $\hat{Z}^{N-N_0}$ from \eqref{eq_CL_traj}, the use of the estimates \eqref{eq_estimate_H1_norm}, and invoking the change of variable formula \eqref{eq_change_var}, we directly infer the existence of a constant $C > 0$, independent of the initial condition, such that the claimed stability estimate \eqref{eq: main stability estimate} holds.
	
	To conclude the proof, it thus remains to show that $N$ can be selected so that $\Theta_1 \preceq 0$ and $\Psi_n \leq 0$ for all $n \geq N+1$. To do so, let us set $\epsilon = 2 N_0^2$ and $\eta = 1/\sqrt{S_{\varphi,N}}$ if  $S_{\varphi,N} \neq 0$, $\eta = N$ otherwise. This implies that $\eta \rightarrow + \infty$ while $\eta S_{\varphi,N} \rightarrow 0$ as $N \rightarrow + \infty$. Hence, for $N$ large enough we have $\eta S_{\varphi,N} \leq 1/2$ and $\lambda_n \geq \lambda_{N+1} > 0$ for all $n \geq N+1$, which implies that
	\begin{align*}
	\Psi_n & \leq \Theta_2 := - \frac{1}{2} \lambda_{N+1} + \frac{3}{2} \nu + 2 \delta , \quad n \geq N+1
	\end{align*}
	with $\Theta_2 \rightarrow - \infty$ as $N \rightarrow + \infty$. Now, since $F$ defined by \eqref{eq_matrix_F} is Hurwitz, we define $P \succ 0$ as the unique solution to the Lyapunov equation $F^\top P + P F + 2 \delta P = -I$. Owing to Lemma~\ref{lemma2}, it can be seen that $\Vert \tilde{C}_1 \Vert = O(1)$, hence $\Vert L \tilde{C}_1 \Vert = O(1)$, as $N \rightarrow + \infty$. Therefore, a result similar to \cite[Lemma in Appendix]{lhachemi2022finite} shows that $\Vert P \Vert = O(1)$ as $N \rightarrow + \infty$. Therefore, we have
	\begin{equation}
	\Theta_1 = \underbrace{\begin{pmatrix} - I + \epsilon S_{1,N} E_1^\top E_1 & P \mathcal{L} \\\mathcal{L}^\top P & -\eta I \end{pmatrix}}_{:= \Theta_{1,p}} + \epsilon S_{2,N} E_2^\top E_2 .
	\end{equation}
	Using the Schur complement for $N$ sufficiently large so that $\eta > 1/2$, we see that $\Theta_{1,p} \preceq -\frac{1}{2} I$ if and only if  $-\frac{1}{2}I + \epsilon S_{1,N} E_1^\top E_1 + \frac{1}{\eta-\frac{1}{2}} P \mathcal{L} \mathcal{L}^\top P \preceq 0$. Noting that $\Vert P \Vert = O(1)$ and $S_{2,N} \rightarrow 0$ as $N \rightarrow + \infty$ while $\Vert E_1 \Vert$ and $\Vert \mathcal{L} \Vert$ are constants independent of $N$, we deduce that $\Theta_{1,p} \preceq -\frac{1}{2} I$ for all $N$ selected to be large enough. In that case, $\Theta_1 \preceq -\frac{1}{2} I + \epsilon S_{2,N} E_2^\top E_2$ for all $N$ selected to be large enough. Since $\Vert E_2 \Vert = O(1)$ and $S_{2,N} \rightarrow 0$ as $N \rightarrow + \infty$, we deduce that $\Theta_1 \preceq 0$ for $N$ large enough. 
\end{proof}

\section{Conclusions}\label{s5} 

This paper discussed the design of an observer-based feedback stabilizing controller for multi-dimensional parabolic type equations governed by diagonalizable second order differential operators. To fix the ideas and to ease the presentation, we focused the developments on the case of three unstable eigenvalues: one of multiplicity one and one of multiplicity two. However, the approach reported in this paper easily extends to any other case with a finite number of unstable modes with arbitrary finite multiplicity. 

To conclude, it is worth to mention that, based on the technique presented in this work, a natural perspective is the study of non-linear multi-dimensional parabolic equations combining the present design method with \cite{lhachemi2021global}, and the study of multi-dimensional  parabolic equations with delays as done in e.g., \cite{lhachemi2020boundary,lhachemi2022boundary}.

\bibliographystyle{plain}
\bibliography{cp}

\end{document}